\documentclass[11pt]{article}
\usepackage{amsmath}
\usepackage{amssymb}
\usepackage{amsthm}
\usepackage{amsmath}
\usepackage{fancyhdr}
\usepackage{amsfonts}

\usepackage[mathscr]{euscript} \usepackage{xfrac} \newcommand{\tabfrac}[3][]{\ensuremath{#1\sfrac{#2}{#3}}} \usepackage{mathrsfs}

\usepackage{graphicx}
\usepackage{a4wide}

\def\tareesidedbox#1{\setbox0=\hbox{$#1$}\dimen0=\wd0 \advance\dimen0 by3pt\rlap{\hbox{\vrule height9pt width.4pt depth2pt \kern-.4pt\vrule height9.4pt width\dimen0 depth-9pt\kern-.4pt \vrule height9pt width.4pt depth2pt}} \relax \hbox to\dimen0{\hss$#1$\hss}}
\def\ho#1{\tareesidedbox{#1}}

\def\other{\eta}
\def\Q{\ensuremath{\mathbb{Q}}}
\def\Z{\ensuremath{\mathbb{Z}}}
\def\M{\ensuremath{\mathscr{M}}}
\def\Num{\ensuremath{\mathscr{N}}}
\def\field#1{\ensuremath{\Q(\zeta_{#1})}}
\def\ring#1{\ensuremath{\Z(\zeta_{#1})}}
\def\Norm{\mathrm{Norm}}
\def\robeqone#1{\ensuremath{2 \cos (\pi / #1)}}
\def\robeqtwo#1{\ensuremath{\sqrt{1 + 4 \cos^2 (\pi / #1)}}}
\def\robnumone{\ensuremath{\sqrt{\frac{5 + \sqrt{13} }{2}}}}
\def\robnumtwo{\ensuremath{\frac{\sqrt7 + \sqrt 3}{2}}}
\def\ournumone{\ensuremath{1+\zeta_{70}+\zeta_{70}^{10}+\zeta_{70}^{29}}}
\newcommand{\e}[2][]{\ensuremath{e^{2 \pi i {#1} / {#2}}}}
\newcommand{\infrac}[2]{\ensuremath{#1/#2}}
\def\WLOG{without loss of generality }
\newcommand\arraySpace{\rule{0pt}{3.1ex}}

\newtheorem{thm}{Theorem}
\newtheorem{lem}{Lemma}
\newtheorem{com}{FC}
\newtheorem{com2}{Comment}
\newtheorem*{defn}{Definition}
\newtheorem*{rem}{Remark}

\title{On the magnitudes of some small cyclotomic integers}
\author{Frederick Robinson \\ \texttt{frobinson@ucla.edu} \and Michael Wurtz \\ \texttt{wurtz@u.northwestern.edu}}

\begin{document}

\maketitle

\begin{abstract}
We prove the last of five outstanding conjectures made by R.M. Robinson from 1965 concerning small cyclotomic integers. In particular, given any cyclotomic integer $\beta$ all
of whose conjugates have absolute value at most $5$, we prove that the \emph{largest} such conjugate has absolute value one of  four explicit types given by 
two infinite classes and
two exceptional cases. We also extend this result by showing that with the addition of one form, the conjecture is true for $\beta$ with magnitudes up to $5+\infrac{1}{25}$.
\end{abstract}

\tableofcontents

\section{Introduction}


In~\cite{Robinson}, Raphael Robinson made a study of small cyclotomic integers, namely, cyclotomic integers $\alpha$ all of whose conjugates lie in $|z| \le R$ for $R = 2$ and $R = \sqrt{5}$. Robinson made a sequence of five conjectures concerning these numbers, four of which were proved by Schinzel~\cite{Schinzel}, Cassels~\cite{Cassels}, and Jones~\cite{Jones1, Jones2}. In this paper, we resolve the final outstanding conjecture. First, we recall the following definition:

\begin{defn}[House]For a cyclotomic integer $\beta$, let the house of $\beta$, denoted ~$\ho{\beta}$, be the largest absolute value of all conjugates of $\beta$.
\end{defn}


Our main result is as follows:

\begin{thm}[Robinson's Conjecture 4 \cite{Robinson}]\label{resultone}
If $\beta$ is a cyclotomic integer with $\ho{\beta}^2 \leq 5$, then $\ho{\beta}$ has one of the forms
\[ \robeqone{N}, \quad \robeqtwo{N}, \]
where $N$ is a positive integer, or else is equal to one of the two numbers
\[ \robnumone, \quad \robnumtwo.\]\end{thm}

Note that these values do actually occur as $\ho{\beta}$ for some cyclotomic integers (with the exception of $N=1$ in the first equation), specifically, for $\beta$ as follows:
$\zeta_N + \zeta^{-1}_N$,
$ \zeta_4 + \zeta_N + \zeta^{-1}_N$,
$1 + \zeta_{13} + \zeta^4_{13}$, and
$\zeta^{-9}_{84} + \zeta^{-7}_{84} + \zeta^{3}_{84} + \zeta^{27}_{84}$.
The first and last numbers on this list are totally real, so $\ho{\beta} = \beta$ in these cases. In studying this problem, we follow the approach of Cassels~\cite{Cassels}, as well as the recent paper of Calegari, Morrison, and Snyder~\cite{Calegari}, where a version of this theorem is proven for totally real $\beta$.

We actually prove the following stronger statement:

\begin{thm} \label{result}  If $\beta$ is a cyclotomic integer with $\ho{\beta}^2 \le 5 + \infrac{1}{25}$, then
either $\ho{\beta}$ is a number on the list above, or
$$\ho{\beta} = \left|\ournumone\right|, \text{ where } \zeta_{70}=e^{2\pi i / 70} $$
\end{thm}

The main result of Cassels~\cite{Cassels} Implies Theorem~\ref{resultone} with \emph{at most finitely many exceptions}.
The methods of Cassels, however, do not lead to a practical algorithm for determining what those exceptions might
be. Indeed, it is noted in~\cite{Calegari} that any exception must lie in $\Z[\zeta_N]$ for
$$N = 4692838820715366441120 
= 2^5 \cdot 3^3 \cdot 5 \cdot 7 \cdot 11 \cdot 13 \cdot 17 \cdot 19 \cdot 23 \cdot 29 \cdot 31 \cdot 37 \cdot 41
\cdot 47 \cdot 53.$$

\medskip

The motivation for this project is twofold. Most naturally, it was desirable
to answer Robinson's conjecture. Robinson was motivated in part by understanding the relationship
between the house of a cyclotomic integer $\alpha$
and the ``complexity'' of such an integer, as for example measured by the number of roots of unity required to represent $\alpha$. Although this problem was qualitatively answered
by Loxton~\cite{Lox}, those arguments are not effective. Another motivation is
to the interaction between the algebraic number theory of cyclotomic fields and the numerology of subfactors of small index, as occurring (for example) in~\cite{J1} and more recently in~\cite{J2}. This was also the motivation for the recent paper~\cite{Calegari}. Although there is no direct application of our result to the indices of subfactors, it is intriguing that the square of the ``exotic''
case $\sqrt{(5+\sqrt{13})/2}$ of Theorem~\ref{resultone} is also the index of the
first exotic subfactor constructed by Aseada and Haagerup~\cite{AH}. 

\subsection{Some Notation}
 
The following is well known:

\begin{lem}[Cyclotomic Integer] A number $\beta \in \Q(\zeta_N)$ is a cyclotomic integer if and only if $\beta \in \Z(\zeta_N)$ for some $N$, i.e. if $\beta$ can be written as a finite sum of roots of unity.
\end{lem}

In light of this, the following definition makes sense:

\begin{defn}[\Num] For a given cyclotomic integer $\beta$, $\Num(\beta)$ is the minimal number of roots of unity whose sum is $\beta$.\end{defn}

Note that given $\alpha$ and $\beta$, we have that $\Num(\alpha) - \Num(\beta) \leq \Num(\alpha \pm \beta) \leq \Num(\alpha) + \Num(\beta)$.

Following Cassels, we also make the following definition:

\begin{defn}[\M] For a given cyclotomic integer $\beta$, $\M(\beta)$ is the arithmetic mean of $|\beta'|^2$ for all conjugates $\beta'$ of $\beta$.
\end{defn}

Note that $|\beta|^2 = \beta \overline{\beta}$ is a cyclotomic integer. Since the Galois group of a cyclotomic extension is abelian, complex conjugation commutes with any automorphism. In particular, $\M(\beta) = \M(\beta')$ for any conjugate $\beta'$ of $\beta$, and moreover $\M(\beta)$ is the (normalized) trace of $|\beta|^2$, and hence lies in $\Q$.

\begin{defn}[Equivalence] Two cyclotomic integers $\alpha$ and $\beta$ are equivalent if $\alpha = \zeta \beta'$ for some $\zeta$ a root of unity and $\beta'$ a conjugate of $\beta$. We write $\alpha \equiv \beta$. \end{defn}

Since every root of unity has absolute value one, it follows that $\M(\zeta \gamma) = \M(\gamma)$ for any root of unity $\zeta$. In particular, if $\alpha \equiv \beta$, then $\M(\alpha) = \M(\zeta \beta') = \M(\beta') = \M(\beta)$.

\begin{defn}[Minimal Cyclotomic Integer] A cyclotomic integer $\beta$ is minimal if $\beta \in \field{N}$, and there is no equivalent $\beta' \in \field{N'}$ with $N' < N$.\end{defn}

Since $\ho{\beta} = \ho{\beta'}$, it suffices to prove the theorem to consider all minimal cyclotomic integers.

\begin{defn}[$\zeta_N$] We always mean a primitive $N$th root of unity by $\zeta_N$, not any $N$th root of unity.\end{defn}

\section{Some Preliminary Results}

\subsection{\texorpdfstring{Properties of \M}{Properties of M}}
\begin{rem}$\Num(\beta) = 1$ if and only if $\M(\beta) = 1$. This follows from~\cite{Kronecker}.
\end{rem}

\begin{lem}
If $\Num(\beta)=2$, either $\M(\beta) \geq \infrac{15}{8}$, or $\M(\beta) = \infrac{3}{2}, \infrac{5}{3}, \infrac{7}{4}, \infrac{9}{5}$, or $\infrac{11}{6}$. The first four values occur only when $\beta$ is equivalent to $1+\zeta_N$ for $N = 5$, $7$, $30$, or $11$ respectively, and $\infrac{11}{6}$ occurs only for $N =13$ or $42$.
\end{lem}

\begin{proof} The sum of two roots of unity is equivalent to $1 + \zeta_N$ for some $N$. 
One computes directly that $\M(1+\zeta_N) = 2(1+\infrac{\mu(N)}{\varphi(N)})$, where $\mu$ is the M\"obius $\mu$-function and $\varphi$ is Euler's totient function, from which the result follows (cf.~\cite{Calegari} Remark 9.0.2). 
\end{proof}

\begin{rem}[Cassels' Lemma 3 \cite{Cassels}]
If $\Num(\beta) \geq 3$, then $\M(\beta) \geq 2$.
\end{rem}

\begin{rem}[Cassels' section 3 \cite{Cassels}] If $\beta \in \ring{N}$, and $p^n$ exactly divides $N$, then we can write $\beta$ as a sum of products of ${p^n}^{th}$ roots of unity with $\other_j \in \ring{N/p}$. Write $\beta = \sum_{j=0}^{p-1} \zeta_{p
^n}^j \other_j$, and let X be the number of non-zero terms in the summation. Let $\alpha_i$, $1 \leq i \leq X$, refer to the $X$ nonzero $\other_j$.

If $p$ exactly divides $N$, note that this representation is unique up to adding a constant to all $\other_i$. We have the equality
\begin{equation}\label{exactEqn}(p-1) \M(\beta) = (p-X) \sum_{i=1}^X \M (\alpha_i) + \sum_{1\leq i <j \leq X}  \M(\alpha_i-\alpha_j).\end{equation}
On the other hand, if $n>1$, then this representation is unique. In this case, we have the equality
\begin{equation}\label{squareEqn}\M(\beta) = \sum_{i=1}^X \M(\alpha_i).\end{equation}
\end{rem}

\subsection{Conjugation}

Throughout the paper, in many cases we will need to show for $\beta$ the sum of two given cyclotomic integers, that $\ho{\beta}^2 > 5+\infrac{1}{25}$, and thus $\beta$ is not an exception to the theorem. One common method of proving this is as follows:

\begin{lem}\label{conjLemma}
Suppose $\beta$ is equivalent to $\alpha + \zeta_{p^n} \gamma$, where $\alpha \in \field{M'}$ and $\gamma \in \field{M''}$. Let $m$ be the largest integer such that $\zeta_{p^m} \in \field{M'}$ or $\field{M''}$. Then if $m<n$,
\begin{equation} \ho{\beta}^2 \geq |\alpha|^2 + |\gamma|^2 + 2 \ |\alpha| \cdot |\gamma| \cdot \cos(\theta) \end{equation}

where 
\[ \theta = \left\{
\begin{array}{ll}
2 \pi / p^n & \mathrm{if\ } m=0 \\
\pi / p^{n-m}& \mathrm{if\ } m>0.
\end{array}
\right.\]

Moreover, if $(M', M'')=1$, then
\begin{equation} \ho{\beta}^2 \geq \ho{\alpha}^2 + \ho{\gamma}^2 + 2 \ \ho{\alpha} \cdot \ho{\gamma} \cdot \cos(\theta).\end{equation}

\end{lem}

\begin{proof}
By assumption on $m$ and $n$, there exists a Galois automorphism sending  $\zeta_{p^n}$ to $\zeta_{p^n}^i$ and fixing $\alpha$ and $\gamma$ as long as $(i,p) = 1$ and $\zeta_{p^m} = \zeta_{p^m}^i$, i.e. when $i \equiv 1 \mod p^m$. If $m=0$, we may conjugate  $\zeta_{p^n}$ to any other \emph{primitive} $p^n$-th root of unity. The largest angle between two adjacent primitive $p^n$-th roots of unity is $2 \cdot 2 \pi / p^n$, so we can place the argument of $\zeta^i_{p^n} \gamma$  to within $2 \pi / p^n$ of the argument of $\alpha$. If $m>0$, then there are $p^{n-m}$ equally spaced primitive $p^n$-th roots of unity that are congruent to $1 \hspace{-1.5mm} \mod p^m$. We can then guarantee that some conjugate of $\beta$ is $\alpha + \zeta_{p^n}^i \gamma$, where the difference in arguments between $\alpha$ and $\zeta_{p^n}^i \gamma$ is at most $\pi / p^{n-m}$.

For the second claim, if $(M', M'')=1$, then we may \emph{simultaneously} conjugate
$\alpha$ and $\gamma$ to  their largest conjugate, and then apply the first part of the Lemma.
\end{proof}

\subsection{A Note on Computational Accuracy}
In several places we have verified results through the use of a computer. For example, given $\beta$, we wish to know if $\ho{\beta}$ is equal to some $\gamma$ from theorem~\ref{result}. We show, that by computing $\ho{\beta}$ to a necessary degree of accuracy, we can claim that $\ho{\beta}$ is equal to $\gamma$, and not just very near to it.

\begin{lem} Suppose $\beta$ is a cyclotomic integer, $\gamma$ is on the list of theorem~\ref{result}, and $k = [\field{N} : \Q] = \varphi(N)$, where $\beta, \gamma \in \field{N}$. If $|\ho{\beta} - \gamma| < (10+\infrac{1}{25})^{-k}$, then $\ho{\beta} = \gamma$.
\end{lem}
\begin{proof}
Let $\delta = |\ho{\beta} - \gamma|$, then $\delta$ is also a cyclotomic integer in $\field{N}$ and $\delta$ has at most $k$ conjugates. Denote the conjugates by $\delta_1, \dots, \delta_i$ with $\delta_1=\delta$. As all conjugates of $\ho{\beta}$ and $\gamma$ have magnitude at most $5+\infrac{1}{25}$, all conjugates of $\delta$ have magnitude at most $10+\infrac{2}{25}$. Then $|\Norm(\delta)| = |\delta_1 \cdots \delta_i| \leq \delta (10+\infrac{2}{25})^{k-1} < 1$. $|\Norm(\delta)| < 1$ if and only if $\Norm(\delta)=0=\delta$, so $\ho{\beta} = \gamma$.
\end{proof}

\subsection{Theorem~\ref{result} when \texorpdfstring{$\Num(\beta) \leq 3$}{N(B) is less than or equal to 3}}

In this section, we recall known results that allow us to deduce Theorem~\ref{result}  in the special
case when $\Num(\beta) \le 3$:
\begin{enumerate}
\item
If $\Num(\beta)=1$, then $\ho{\beta}=1=\robeqone{3}$.
\item
If $\Num(\beta)=2$, then $\beta \equiv 1+\zeta_n$ for some $n$ and $\ho{\beta}=2|\cos(\pi / n)|$.
\item
If $\Num(\beta)=3$, Jones'~\cite{Jones2} Theorem 2 states that if $\ho{\beta} \leq 1 + \sqrt{2}$, then $\beta$ is equivalent to $1+\zeta_n-\zeta_n^{-1}$, $1 \pm i + \zeta_n$, or one of $15$ numbers that he lists.

In the first case, $\beta$ equivalent to $1+\zeta_n-\zeta_n^{-1}$, we have that $\ho{\beta}$ is equal to $\robeqtwo{M}$ where the value of $M$ depends on $n$ in a slightly subtle way. In particular,
\[ M(n) = \left\{
\begin{array}{ll}
2n & \mathrm{if\ } n \text{ is odd} \\
n & \mathrm{if\ } n/2 \text{ is odd} \\
n/4 & \mathrm{if\ } n/4 \text{ is odd} \\
n/2 & \mathrm{if\ } n/4 \text{ is even}.
\end{array}
\right.\]

In the second case, $\beta$ is equivalent to $1 \pm i + \zeta_n$. Lemma~\ref{conjLemma} proves that if $n$ does not divide $2^4 \cdot 3 \cdot 5 \cdot 7$, then $\ho{\beta}>\sqrt{5+\infrac{1}{25}}$ (by letting $\alpha = 1 + i$). There are then $40$ divisors of $2^4 \cdot 3 \cdot 5 \cdot 7$ that were checked computationally.

We checked each number in the third case, and all were equal to a form from Robinson.
\end{enumerate}

\section{An upper bound for \texorpdfstring{$\M(\beta)$}{M(B)}}

Many of our arguments are based on the following Lemma:

\begin{lem}\label{134}If $\beta$ is a cyclotomic integer with $\ho{\beta}^2 \leq 5 + \infrac{1}{25}$, then $\M(\beta) < \infrac{13}{4}$ or $\ho{\beta} = \robeqtwo{N}$ for some $N$.
\end{lem}

\begin{rem} \emph{One should compare this with  Lemma~5.1.1 of~\cite{Calegari}, where, assuming the slightly weaker condition $\ho{\beta} \leq \infrac{76}{33}$, it is shown that $\M(\beta) < \infrac{23}{6}$. The significant improvement ($\infrac{23}{6} = \infrac{13}{4} + \infrac{7}{12}$) in our upper bound for $\M(\beta)$ (at the cost of a stronger bound on $\ho{\beta}$) is what allows us to push the methods of Cassels and~\cite{Calegari} to prove Robinson's conjecture.}
\end{rem}

\begin{proof}
Let $P_i$ and $\alpha_i$ be as below (note that all $P_i$ are irreducible over \Z, and their roots are real and positive):
\[
\begin{array}{l|l|r|l}
i & P_i & 1000\alpha_i & N\\
\hline
1 & x-3 & 110 & 4\\
2 & x-4 & 530 & 6\\
3 & x-5 & 620 & 1\\
4 & x^2-6x+6 & 18 & 12\\
5 & x^2-6x+7 & 28 & 8\\
6 & x^2-7x+11 & 194 & 10\\
7 & x^3-10x^2+31x-29 & 130 & 14\\
8 & x^4-13x^3+58x^2-98x+41 & 45 & D_8\\
9 & x^4-13x^3+59x^2-107x+61 & 40 & 15
\end{array}
\]
Let $f(x) =\infrac{13}{4} - x - \sum \alpha_i \log|P_i(x)|$. We claim that $f(x)$ is positive for all values of $x$ in $[0,5+\infrac{1}{25}]$ where it is defined (there are many asymptotes where $f(x) \to +\infty$). Note that $f$ is defined everywhere that is not a root of some $P_i$.
\[
\includegraphics[height=210px]{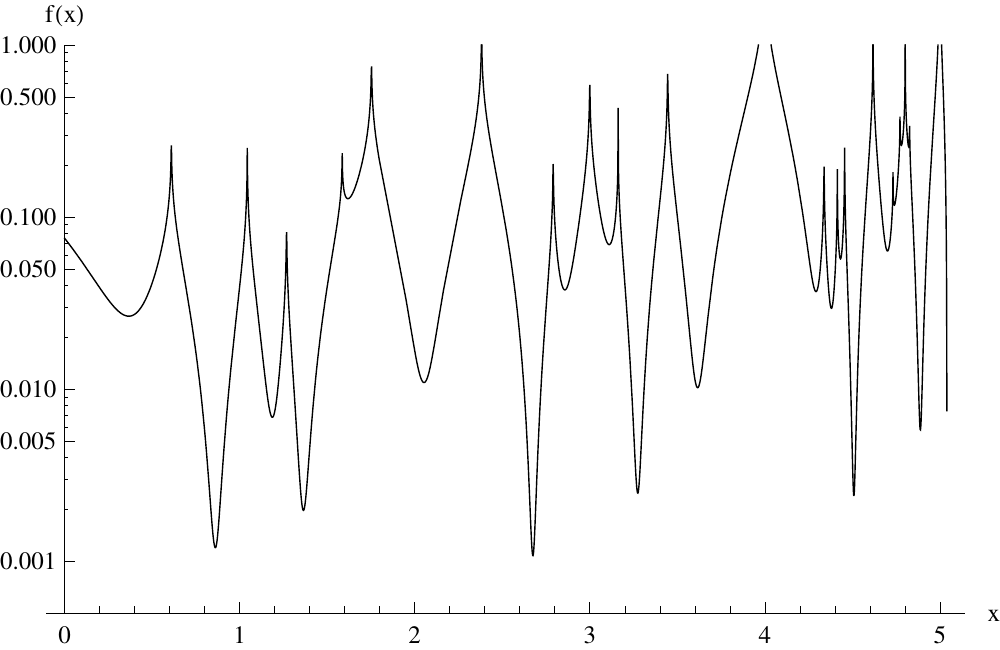}
\]

The derivative of $f(x)$ has $14$ real zeroes in $[0,5+\infrac{1}{25}]$, at which all of $f$ is positive. Also, $f$ is positive at 0 and $5+\infrac{1}{25}$. So $f$ is positive everywhere on $[0,5+\infrac{1}{25}]$ where it is defined.

Now take any non-zero cyclotomic integer $\beta$ with $\beta \bar\beta = \ho{\beta}^2 \leq 5+\infrac{1}{25}$.

If $\ho{\beta}^2$ is equivalent to a root of some $P_i$, note two things: it cannot be $P_8$, as that has a non-abelian Galois group which would imply that $\beta$ is not a cyclotomic integer. Furthermore $\ho{\beta}^2$ is the largest root of $P_i$. All largest roots of $P_i$, $i \neq 8$, are squares of \robeqtwo{N} for $N$ as shown in the above table.

If $\beta \bar\beta$ is not equivalent to a root of any $P_i$, let $x_j$, $1 \leq j \leq n$, be the conjugates of $\beta \bar\beta$. Note that the conjugates of $\beta \bar\beta$ are $\beta' \bar{\beta'}$ for $\beta'$ the conjugates of $\beta$. Then $0 < x_j \leq 5+\infrac{1}{25}$ and $P_i(x_j) \neq 0$ for any $i, j$, so we have
\begin{align*}\sum^n_{j=1} f(x_j) &> 0\\
\sum_{j=1}^n \left( \frac{13}{4} - x_j - \sum_i \alpha_i \log |P_i(x_j)| \right) &> 0\\
\frac{13}{4} n - \sum_{j=1}^n x_j - \sum_i \left(a_i \sum_{j=1}^n \log|P_i(x_j)|\right) &> 0\\
\frac{13}{4}n -n \M(\beta) - \sum_i a_i \log \left| \prod_{j=1}^n P_i (x_j)\right| &> 0\\
\frac{13}{4} n - n \M(\beta) &> \sum_i a_i \log \left|\prod_{j=1}^n P_i(x_j)\right|\\
\frac{13}{4} n - n \M(\beta) &> \sum_i a_i \log \left|\Norm(P_i(\beta\bar\beta))\right|\\
\frac{13}{4}n - n\M(\beta) &> 0\\
\frac{13}{4} & > \M(\beta)
\end{align*}
\end{proof}

\section{\texorpdfstring{If $\beta \in \field{N}$, then $N$ or $N/2$ is squarefree}{If B is in Q(Zn), then N or N/2 is squarefree}}
\begin{lem}\label{psquared} Suppose $\beta \in \field{N}$ is a minimal exception to Theorem \ref{result}. If $p^2$ divides $N$, then $p = 2$ and $4$ exactly divides $N$.
%

\end{lem}
Suppose towards a contradiction that $p^n$ exactly divides $N$, with $n \geq 2$ and $p^n \neq 4$.
Write $\beta = \sum_{i=0}^{p-1} \zeta_{p^n}^j \other_j$, with $\other_j \in \field{N/p}$.
We refer to this as the $p$-decomposition of $\beta$.  Let $\alpha_i$ be the $X$ nonzero $\other_i$. We have by Cassels \cite{Cassels} that $\M(\beta) = \sum \M(\alpha_i)$, so by Lemma~\ref{134}, $\sum \M(\alpha_i) < \infrac{13}{4}$. $X$ must be $2$. $X=1$ would mean $N$ is not minimal, $X=3$ would mean $\Num(\beta)=3$, and $X>3$ would mean $\M(\beta) \geq 4$.

Let $\beta = \alpha + \zeta_{p^n}\gamma$, and assume  \WLOG that $\M(\alpha) \leq \M(\gamma)$. Then $\M(\alpha) \leq \infrac{13}{8}$, so $\M(\alpha) = 1$ or $\infrac{3}{2}$.

\subsection{\texorpdfstring{$\M(\alpha) = 1$}{M(a)=1}}

Recall that $\Num(\beta) > 3$, so $\Num(\gamma) \geq 3$.

Assume  \WLOG (by multiplying $\beta$ by a root of unity) that $\alpha = 1$. We know that $2 \leq \M(\gamma) < \infrac{9}{4}$.

\begin{itemize}
\item First assume that $\ho{\gamma}^2 > 2$, then by Cassels' corollary to Lemma 5, we have $\ho{\gamma}^2 \geq 3$. If $p \geq 3$, then by Lemma~\ref{conjLemma}, $\ho{\beta}^2 \geq 4+\sqrt{3}$.

In the case of $2^n$, $n > 2$, write $\gamma = \gamma' + \zeta_{2^{n-1}} \gamma''$, with $\gamma', \gamma'' \in \field{N/4}$. $\M(\gamma')+\M(\gamma'') = \M(\gamma) < 2 \infrac{1}{4}$, so either both $\gamma'$ and $\gamma''$ are roots of unity and $\beta$ is 3 roots of unity, or one of $\gamma'$ or $\gamma''$ are 0.
The latter case implies $\beta \equiv 1 + \zeta_{2^n}^i(\gamma' + \gamma'')$, and by Lemma~\ref{conjLemma}, $\ho{\beta}^2 \geq 4+\sqrt{6}$.

\item The other case is if $\M(\gamma) = \ho{\gamma}^2 = 2$. By Cassels' Lemma 6, $\gamma$ is equivalent to one of $(-1 + \sqrt{-7})/2 \equiv 1+\zeta_7+\zeta_7^3$ or $(\sqrt{5}+\sqrt{-3})/2 \equiv \zeta_3-\zeta_5-\zeta_5^{-1}$. We break down into cases as follows:

\begin{itemize}
\item $p^n=3^2$ and $\gamma \equiv 1+\zeta_7+\zeta_7^3$

then $\theta \leq 2\pi/9$ and $\ho{\beta}^2 > 5.1667$.

\item $p^n=3^2$ and $\gamma \equiv \zeta_3-\zeta_5-\zeta_5^{-1}$

Here, we have $\gamma = \zeta_m \cdot (\zeta_3^j-\zeta_5^k-\zeta_5^{-k})$, and so, after multiplying $\beta$ by some root of unity, we may assume $\beta$ is of the form $1+\zeta_{3^2}^i \cdot \zeta_m^l \cdot (\zeta_3^j-\zeta_5^k-\zeta_5^{-k})$ for some values of $i,j,k,l$. If $2^4$, $5^2$, or any prime greater than $5$ divides $m$, we may conjugate $\beta$ by Lemma~\ref{conjLemma}. We may also assume (by changing $i$) that $3$ does not divide $m$. This limits $m$ to $8$ possible values. We may conjugate $\zeta_m$ such that $l=1$. There are then $384 = 2 \cdot 4 \cdot 6 \cdot 8$ possibilities for $\beta$. Computation reveals that all of these have $\ho{\beta}^2 > 5.094$.

\item $p^n=2^3$ and $\gamma \equiv \zeta_3-\zeta_5-\zeta_5^{-1}$ or $\gamma \equiv 1+\zeta_7+\zeta_7^3$

Then $\beta$ is of the form $1+\zeta_{2^3}^i \cdot \zeta_m^l \cdot \gamma'$ for some $i,l$, and $\gamma'$ a conjugate of $\zeta_3-\zeta_5-\zeta_5^{-1}$ or $1+\zeta_7+\zeta_7^3$. Reasoning as above, $m$ divides $3^2 \cdot 5 \cdot 7$. There are then $12$ possible values for $m$. There are $672 = 4 \cdot 12 \cdot (8+6)$ possibilities for $\beta$. Computation reveals that all of these have $\ho{\beta}^2 > 5.0489$.


\item In all other cases, $\theta \leq \pi/5$. Hence $\ho{\beta}^2 \geq 3+2\sqrt{2}\cos(\pi/5) \approx 5.28825$.
\end{itemize}
\end{itemize}

\subsection{\texorpdfstring{$\M(\alpha) = \infrac{3}{2}$}{M(a)=3/2}}
Note that $\M(\gamma) < \infrac{13}{4}-\infrac{3}{2}=\infrac{7}{4}$.

\begin{itemize}
\item $\M(\gamma) = \infrac{3}{2}$

$\alpha$ and $\gamma$ are both equivalent to $1 + \zeta_5$. Let $\zeta_5 = \e{5}$, then by conjugating and multiplication by a root of unity, assume \WLOG that $\beta = (1 + \zeta_5) + \varrho(\zeta_5^i + \zeta_5^j)$ for some root of unity $\varrho$. If the difference between $i$ and $j$ (mod 5) is 2 or 3, then $\beta$ is equivalent to $(\zeta_5 + \zeta_5^4) + \varrho' (\zeta_5^2 + \zeta_5^3)$. If the difference $i-j \mod 5$ is $1$ or $4$, then $\beta$ is equivalent to $\alpha + \zeta_{p^n} \gamma$ with $|\alpha| = |\gamma|=(1+\sqrt{5})/2$. Then by Lemma \ref{conjLemma}, regardless of $p^n$ we have $\theta \leq \pi/4$, and $\ho{\beta}^2 \geq (2+\sqrt{2}) (3+\sqrt{5})/2 \approx 8.93853$.

\item $\M(\gamma) = \infrac{5}{3}$

$\alpha$ is equivalent to $1 + \zeta_5$, and $\gamma$ is equivalent to $1 + \zeta_7$. Again by Lemma \ref{conjLemma}, regardless of $p^n$ we have $\theta \leq \pi/4$, and $\ho{\beta}^2$ is even larger than the preceding case.
\end{itemize}

\section{\texorpdfstring{If $\beta \in \field{N}$, then $N$ divides $420$}{If B is in Q(Zn), then N divides 420}}

\begin{lem} If $\ho{\beta}^2 < 5.3$, then either $\beta$ is on the list of Theorem~\ref{result}, or $\beta \in \field{420}$.
\end{lem}

First we'll establish some facts that we use throughout. Recall that $X$ refers to the number of nonzero terms in the $p$-decomposition of $\beta$.
 
We have from Cassels' \cite{Cassels} (3.5) and Lemma~\ref{134}
\begin{equation}\label{two}p \geq 11 \Rightarrow X \leq \frac{p-1}{2}.\end{equation}

By equation~\ref{exactEqn}, and since $ \M(\beta) < \infrac{13}{4}$ by Lemma~\ref{134}, we have
\begin{equation}\label{134eqn}\frac{13}{4} (p-1) > (p-X) \sum_i^X  \M(\alpha_i) + \sum_{1 \leq i<j\leq X}  \M(\alpha_i - \alpha_j). \end{equation}

Now let $\beta \in \field{N}$ be a minimal cyclotomic integer that is an exception to Theorem~\ref{result}. Let $p$ be the largest prime dividing $N$, and suppose $p>7$. By Lemma~\ref{psquared}, $p$ exactly divides $N$. We proceed by considering different combinations of $p$ and $X$.

\subsection{\texorpdfstring{$p = 11$}{p=11}}
Note that by equation~\ref{two}, $X \leq 5$.

\subsubsection{\texorpdfstring{$X=2$}{X=2}}\label{p11X2}

By equation~\ref{134eqn}

\[\frac{65}{2} > 9(\M(\gamma)+\M(\alpha)) + \M(\gamma-\alpha).\]

Assume  \WLOG that $\M(\alpha) \leq \M(\gamma)$.
\[\M(\alpha) \leq \frac{65}{36} \Rightarrow \M(\alpha)=1, \frac{3}{2}, \frac{5}{3}, \frac{7}{4}\mathrm{\ or\ }\frac{9}{5}.\]

We consider each possible value of $\M(\alpha)$ below:
\begin{itemize}
\item{$\M(\alpha)=1$.}
As $\Num(\beta) > 3$, $\Num(\gamma) > 2$ and thus $\M(\gamma) \geq 2$. By conjugating we can assume $|\gamma| \geq \sqrt{2}$, then by Lemma~\ref{conjLemma} we have $\ho{\beta}^2 \geq 3+2 \sqrt{2} \cos(2\pi/11) \approx 5.37942$.

\item{$\M(\alpha)=\infrac{3}{2}$.} If $\Num(\gamma) > 2$ the inequality is false, since $\M(\gamma-\alpha) \geq 1$ and $M(\gamma) \geq 2$. If $\Num(\gamma) = 1$ then $\Num(\beta) = 3$. So let $\gamma$ be equivalent to $1+\zeta_n$.

If $n=5$, $\alpha \equiv \gamma \equiv 1+\zeta_5$. As we have previously argued, either $\beta \equiv (\zeta_5 + \zeta_5^4) + \varrho (\zeta_5^2 + \zeta_5^3)$ for some root of unity $\varrho$, or we can conjugate to assume that $|\alpha|=|\gamma|=(1+\sqrt{5})/2$. Then $\ho{\beta}^2 \geq (3+\sqrt{5}) (1+\cos(2\pi/11)) \approx 9.64093$.

If $n$ is coprime to $5$ then $n\geq4$ and by (Lemma~\ref{conjLemma}), with $\theta=2\pi/11$, $\ho{1+\zeta_5}=(1+\sqrt{5})/2$, and $\ho{1+\zeta_n} \geq \sqrt{2}$, we have $\ho{\beta}^2 \geq 8.46802$ (Lemma~\ref{conjLemma}).

If $n$ is divisible by $5$ it must be at least $10$. Conjugate $\beta$ so $|1+\zeta_n| = \ho{1+\zeta_n} \geq \ho{1+\zeta_{10}} = \sqrt{(5+\sqrt{5})/2}$. The smallest conjugate of $1+\zeta_5$ is $(\sqrt{5}-1)/2$. Thus by Lemma~\ref{conjLemma}, with $\theta=2\pi/11$, $\ho{\beta}^2 > 5.9779$ .

\item{$\M(\alpha)=\infrac{5}{3}$.} Again we have $\Num(\gamma) = 2$. Let $\gamma$ be equivalent to $1+\zeta_n$.

If $n=7$ then both $\alpha$ and $\gamma$ are equivalent to $1+\zeta_7$. We may conjugate them simultaneously so neither is the smallest conjugate as follows: let $\zeta_7$ be $\e{7}$. Assume \WLOG that $\alpha = 1 + \zeta_7$, and $\gamma = \varrho(1 + \zeta_7^i)$ for $\varrho$ a root of unity. Then $\beta = (1+\zeta_7)+\zeta_{11}\varrho(1+\zeta_7^i)$. If $i \neq 3,4$ we are done, otherwise, $\beta$ under the conjugation $\zeta_7 \rightarrow \zeta_7^2$ is $(1+\zeta_7^2)+\zeta_{11}\varrho'(1+\zeta_7^{2i})$, which satisfies our requirement. We now have $|\alpha|, |\gamma| \geq |\e[2]{7}|$, and then by Lemma~\ref{conjLemma}, $\ho{\beta}^2 > 6.09385$. 

The case $n$ coprime to $7$ easily follows from the previous case with $\M(\alpha)=\infrac{3}{2}$, since $\ho{1+\zeta_7} > \ho{1+\zeta_5}$.

If $n$ is divisible by $7$, similarly to before, conjugate $\gamma$ to its largest conjugate and then $|\gamma| = \ho{\gamma} \geq |\e{14}|$ and $|\alpha| \geq |\e[3]{7}|$. We then have $\ho{\beta}^2 > 5.66523$

\item{$\M(\alpha)=\infrac{9}{5}$.} Then $\alpha \equiv 1+\zeta_{11}$, but this is impossible as $\zeta_{11} \notin \field{N/11}$.

\item{$\M(\alpha)=\infrac{7}{4}$.} 
From the inequality, $\M(\gamma) = \infrac{7}{4}$ or $\infrac{11}{6}$. However, we see that $\infrac{11}{6}$ makes the inequality false with $\M(\alpha-\gamma) \geq 1$, so $\M(\gamma) = \infrac{7}{4}$.

Recall that $\M(\rho) = \infrac{7}{4} \Rightarrow \rho \equiv 1 + \zeta_{30}$. Conjugate $\alpha$ to $1 + \e{30}$. This will fix $\gamma$ to be some other conjugate, of which the smallest is $1+ \e[13]{30}$. By  Lemma~\ref{conjLemma}, $\ho{\beta}^2 > 5.71638$.

\end{itemize}

\subsubsection{\texorpdfstring{$X=3$}{X=3}}
By equation~\ref{134eqn}
\[\frac{65}{2} > 8 \sum_{i=1}^3 \M(\alpha_i) + \sum_{1 \leq i<j\leq 3}  \M(\alpha_i - \alpha_j).\]
If more than one $\alpha_i$ is not a root of unity, the inequality is false. We may assume that not all three $\alpha_i$ are roots of unity, as the case $\Num(\beta)=3$ is done. Then notice that $\Num(\alpha_i)>2$ again makes the inequality false. So we may assume  \WLOG that $\Num(\alpha_1)=\Num(\alpha_2)=1$, and $\Num(\alpha_3)=2$.

Either the respective \M~values are $(1,1,\infrac{3}{2})$, $(1,1,\infrac{5}{3})$, or $(1,1,\infrac{7}{4})$.

We must calculate $\ho{\beta}$ with $\beta$ of the form $1 + \zeta_{11} \zeta_{420}^i + \zeta_{11}^j \zeta_{420}^k (1+\zeta_n)$ for all $i,j,k$ where $n = 5, 7$, or $30$. Some computation shows that the smallest such $\beta$ is $1+\zeta_{77}+\zeta_{77}^{11}+\zeta_{77}^{55}$ and $\ho{\beta} > 5.761$.

\subsubsection{\texorpdfstring{$X=4$}{X=4}}
By equation~\ref{134eqn}
\[\frac{65}{2}>7 \sum_{i=1}^4  \M(\alpha_i) + \sum_{1 \leq i<j\leq 4}  \M(\alpha_i - \alpha_j). \]
If any $\alpha_i$ is not a root of unity, this inequality is false.

Therefore each $\alpha_i$ is a root of unity and
\[\frac{9}{2} > \sum_{1 \leq i<j\leq 4}  \M(\alpha_i - \alpha_j),\]
which implies there are at most 2 distinct roots of unity.

There remain 3 cases after conjugation. $(\alpha_1,\alpha_2,\alpha_3,\alpha_4)$ is equivalent to one of the following, for $\zeta$ some root of unity: \[(1,1,\zeta,\zeta) \mathrm{\ or\ } (1,1,1,\zeta) \mathrm{\ or\ } (1,1,1,1).\]
In both cases with $\zeta$, we must have $\M(1-\zeta)=1$ or else the inequality is false. Thus $\zeta = \zeta_6$.

The only such $\beta$ with $\ho{\beta}^2<6$ of the above form are below
\[\begin{array}{c|c|c}
(\alpha_1, \alpha_2, \alpha_3, \alpha_4)& \beta& \ho{\beta}\\
\hline \arraySpace
(1,1,1,1)&1+\zeta_{11}+\zeta_{11}^2+\zeta_{11}^5&\robeqtwo{11}\\ \arraySpace
(1,1,1,\zeta_6)& \zeta_6+\zeta_{11}+\zeta_{11}^3+\zeta_{11}^8& \robeqtwo{33}\\ \arraySpace
(1,1,\zeta_6,\zeta_6)& \zeta_6+\zeta_6\zeta_{11}+\zeta_{11}^3+\zeta_{11}^9&\robeqtwo{22}
\end{array}\]

\subsubsection{\texorpdfstring{$X=5$}{X=5}}\label{above}
By equation~\ref{134eqn}
\[\frac{65}{2}> 6 \sum_{i=1}^5 \M(\alpha_i) + \sum_{1 \leq i<j\leq 5}  \M(\alpha_i - \alpha_j).\]
If any $\alpha_i$ is not a root of unity it has $\M(\alpha_i)\geq \infrac{3}{2}$ and this inequality is false. So all $\alpha_i$ are roots of unity, and
\[\frac{5}{2} >  \sum_{1 \leq i<j\leq 5}  \M(\alpha_i - \alpha_j).\]
However, if there exists $\alpha_i\neq\alpha_j$ then the above inequality is false, so we may assume  \WLOG that $\alpha_i=1$ for all $i$.

One can compute every 
\[\beta=1+\zeta_{11}+\zeta_{11}^a+\zeta_{11}^b+\zeta_{11}^c\]
with $a, b, c$ distinct and not equal to $0$ or $1$, and $\ho{\beta}<\sqrt{5}+.1$. They are all equivalent to
\[1+ \zeta_{11}+\zeta_{11}^2+\zeta_{11}^4+\zeta_{11}^7 .\]
Which has
\[\ho{\beta} = \robeqone{6} = \robeqtwo{4}. \]

\subsection{\texorpdfstring{$X=2$}{X=2}}

Let $\beta = \alpha + \zeta_p \gamma$, with $\alpha, \gamma \in \field{N/p}$ and $p \geq 13$. By equation~\ref{134eqn},
\begin{align*}
\frac{13}{4} (p-1) &> (p-2)(\M(\alpha)+\M(\gamma))+\M(\alpha-\gamma) \\
\frac{13}{4} \cdot \frac{p-1}{p-2} &> \M(\alpha)+\M(\gamma)+\frac{\M(\alpha-\gamma)}{p-2} \\
\frac{39}{11} &> \M(\alpha)+\M(\gamma)
\end{align*}

From here, the reasoning follows almost exactly as in Section~\ref{p11X2}. As $p>11$, any argument based on Lemma~\ref{conjLemma} is still valid, as $\theta$ will be smaller. There are two cases where we need to change the argument: $1+\zeta_{11}$ can appear, and the difference term $\M(\alpha-\gamma)$ may be larger.

We can still assume $\M(\alpha) \neq \infrac{9}{5}$, now because $\frac{1}{2} \cdot 39/11 < 9/5$.

In the $\M(\alpha) = \infrac{7}{4}$ case, $\M(\gamma) = \infrac{11}{6}$ or $\infrac{9}{5}$ is still not possible: now by the restriction on $\M(\alpha)+\M(\gamma)$ instead of the other reasons.

\subsection{\texorpdfstring{$p = 13$}{p=13}}
\subsubsection{\texorpdfstring{$X=3$}{X=3}}

By equation~\ref{134eqn}
\[39 > 10 \sum \M(\alpha_i) + \sum \M(\alpha_i - \alpha_j). \]
We may assume that the \Num~values $(\Num(\alpha_1), \Num(\alpha_2), \Num(\alpha_3))$ are $(1, 1, 2)$. If they are less, $\Num(\beta)=3$, and if they are more, the inequality is false.

The \M~values must be $(1, 1, \infrac{3}{2})$ or $(1, 1, \infrac{5}{3})$ for the inequality to hold.
\begin{itemize}
\item $(1, 1, \infrac{3}{2})$

$\sum \M(\alpha_i - \alpha_j)<4$. Neither $\Num(\alpha_3-\alpha_1)=3$ nor $\Num(\alpha_3-\alpha_2)=3$. If so, then either $\alpha_2-\alpha_1=0$ and $0+2+2 \geq 4$, or $\Num(\alpha_2-\alpha_1)=1$ and $1+1+2 \geq 4$.

Assume \WLOG that $\alpha_1 = 1$. Because $\Num(\alpha_3 - \alpha_1) \leq 2$, there is some cancellation occurring in the difference $\alpha_3 - \alpha_1$. In particular, $\alpha_3$ must be equal to $1+\zeta_5$, $\zeta_5+\zeta_5^i$, or $\zeta_6+\zeta_6\zeta_5$. We divide into cases based on $\Num(\alpha_3 - \alpha_1)$, and employ a result of Mann~\cite{Mann} (see also Poonen and Rubinstein~\cite{Poonen}). For small $n$, he classified vanishing sums of $n$ roots of unity. For $n<6$, these must be sums of groups comprised of equally spaced roots of unity.
\begin{itemize}
\item $\Num(\alpha_3 - \alpha_1) = 0$, then we have a vanishing sum of 3 roots of unity, which is impossible when two of them differ by a fifth root of unity.
\item $\Num(\alpha_3 - \alpha_1) = 1$, then we have a vanishing sum of 4 roots of unity, and by Poonen, it must consist of two groups of 2 roots of unity each of whose sum vanishes. $\alpha_3 = 1 + \zeta_5$.
\item $\Num(\alpha_3 - \alpha_1) = 2$, then we have a vanishing sum of 5 roots of unity, and by Poonen, it must be a primitive vanishing sum of 5 roots of unity, or is two vanishing sums, one of 2 roots of unity and one of 3 roots of unity. If we are in the 5 case and $\alpha_3 = \zeta_5 + \zeta_5^i$ or we are in the 2-3 case and $\alpha_3 = \zeta_6 + \zeta_6\zeta_5$.
\end{itemize}
So, we may assume \WLOG that $(\alpha_1, \alpha_2, \alpha_3)$ is one of the following: $(1, \zeta_5, \zeta_5^i+\zeta_5^j)$, $(1, \zeta_5, \zeta_6+\zeta_5^i\zeta_6^j)$, or $(1, \zeta_6, \zeta_6^i+\zeta_5\zeta_6^j)$ for some $i,j$.

We compute the house of all $\beta$ with restrictions from above, and in all cases, $\ho{\beta}^2>5.66$.

\item $(1, 1, \infrac{5}{3})$

$\sum \M(\alpha_i - \alpha_j)<\infrac{7}{3}$. $\alpha_1=\alpha_2$, and $\alpha_3-\alpha_1$ is a root of unity. So we may assume that $\alpha_1=\alpha_2=1$ and $\alpha_3=1+\zeta_7$. The smallest such $\beta$ is $1+\zeta_{13}+\zeta_{13}^2(1+\zeta_5)$ with $\ho{\beta}^2>10$.
\end{itemize}

\subsubsection{\texorpdfstring{$X=4$}{X=4}}
We proceed as in Section~\ref{above}. By equation~\ref{134eqn}
\[39> 9 \sum_{i=1}^4 \M(\alpha_i)\]
implies that $\alpha_i$ are all roots of unity, and 
\[ 3 >  \sum_{1 \leq i<j\leq 4}  \M(\alpha_i - \alpha_j) \]
implies that they are all the same root of unity. Thus $\beta \in \field{13}$ is a sum of $4$ roots of unity.

One can verify that the the only such $\beta$ with $\ho{\beta}<\sqrt{5} + .1$ is
\[\beta = 1 + \zeta_{13}+ \zeta_{13}^3+ \zeta_{13}^9,\]
with
\[\ho{\beta} = \robeqone{6} = \robeqtwo{4}.\]

\subsubsection{\texorpdfstring{$X\geq 5$}{X is greater than or equal to 5}}
By equation~\ref{two}, $p=13$ implies that $X \leq 6$. Equation \ref{134eqn} gives us
\[\frac{13}{4} \cdot 12 > (13-X)X,\] 
which is false for $X=5,6$.

\subsection{\texorpdfstring{$X = 3$}{X=3}}

\subsubsection{\texorpdfstring{$p=17$}{p=17}}
By equation~\ref{134eqn}
\[52 > 14 \sum  \M(\alpha_j)  +  \sum\M(\alpha_i - \alpha_j).\]
So we may assume  \WLOG that $\alpha_1$ and $\alpha_2$ are both roots of unity. Also, assume $\alpha_3$ is not a root of unity, since this case has been done already. We may conclude that $\M(\alpha_3)=\infrac{3}{2}$, otherwise the inequality is false.

Now,
\[3 >   \M(\alpha_1 - \alpha_2) + \M(\alpha_1 - \alpha_3) + \M(\alpha_2 - \alpha_3). \]

For this to hold, we must have $\alpha_1=\alpha_2$ and $\alpha_3-\alpha_1$ a root of unity. We may assume \WLOG that $(\alpha_1, \alpha_2, \alpha_3)=(1, 1, 1+\zeta_5)$, and then that $\beta$ is equivalent to
\[\beta = 1 + \zeta_{17}+\zeta_{17}^j(1+\zeta_5)\]
for some $j$. One can verify that the smallest such such $\beta$ is $1 + \zeta_{17}+\zeta_{17}^5(1+\zeta_5)$ with $\ho{\beta}^2>9$.

\subsubsection{\texorpdfstring{$p=19$}{p=19}}
By equation~\ref{134eqn}
\[\frac{117}{2}  > 16 \sum  \M(\alpha_i) + \sum\M(\alpha_i - \alpha_j). \]
As in the previous section, $(\M(\alpha_1), \M(\alpha_2), \M(\alpha_3))=(1, 1, \frac{3}{2})$. Then
\[\frac{5}{2} >   \M(\alpha_1 - \alpha_2) + \M(\alpha_1 - \alpha_3) + \M(\alpha_2 - \alpha_3) \]
and we may assume \WLOG that $(\alpha_1, \alpha_2, \alpha_3) = (1, 1, 1+\zeta_5)$. $\beta$ is equivalent to 
\[\beta = 1 + \zeta_{19}+\zeta_{19}^j(1+\zeta_5)\]
for some $j$. One can verify that the smallest such such $\beta$ is $1 + \zeta_{19}+\zeta_{19}^5(1+\zeta_5)$ with $\ho{\beta}^2>10$.

\subsubsection{\texorpdfstring{$p \geq 23$}{p is greater than or equal to 23}}
In this case, all $\alpha_i$ must be roots of unity. Otherwise, this contradicts equation~\ref{134eqn}. Thus $\Num(\beta)=3$ and there are no exceptions to theorem~\ref{result}.

\subsection{\texorpdfstring{$X \geq 4$ and $p \geq 17$}{X is greater than or equal to 4, p is greater than or equal to 17}}
By equation~\ref{134eqn}
\[\frac{13}{4}(p-1)>(p-X)X,\]
which is false for $X \geq 4$ and $p \geq 17$ when $X \leq (p-1)/2$ as required by equation \ref{two}. We can see this, as 
\[\frac{d}{dX}(pX-X^2) = p-2X,\]
so for $x<p/2$, $pX-X^2$ increases with $x$. Thus the minimal value for $(p-X)X$ in the region is at $X=4$, but
\[\frac{13}{4}(p-1)>(p-4)4\]
is false for $p \geq 17$.

\section{There are no exceptions in \texorpdfstring{$\field{420}$}{Q(Z 420)}}
\begin{lem} Theorem~\ref{result} holds for $\beta \in \field{420}$.

\end{lem}
We have computed all $\beta \in \field{420}$ with $\Num(\beta) \leq 6$ as follows:  \WLOG we assume that the first root is 1, the second root $\zeta_{420}^i$ has $i$ dividing $420$ (or equal to 0), and the other roots $\zeta_{420}^j$ have $(420, j) \geq i$. No exceptions were found, thus we know that any exceptions $\beta$ must have $\Num(\beta) > 6$.

Write $\beta$ as $\sum_{i=0}^{4} \zeta_5^i \other_i$ with $\other_i \in \field{84}$. Let $X$ be the minimal number of nonzero $\other_i$ that can represent $\beta$ in this way, and let $\alpha_i$ be these nonzero $\other_i$.

In the below cases we make use of several facts about $\alpha \in \field{84}$:
\begin{itemize}
\item If $\Num(\alpha) = 2$, then $\M(\alpha) \geq \infrac{5}{3}$, as $1+\zeta_5$ cannot appear.
\item If $\Num(\alpha)=4$, then $\M(\alpha) \geq \infrac{5}{2}$, by~\cite{Calegari} 7.0.8.
\item If $\Num(\alpha) \geq 5$, then $\M(\alpha) \geq \infrac{17}{6}$, by~\cite{Calegari} 7.0.8.
\end{itemize}

In each of the following cases, we demonstrate a contradiction to equation~\ref{134eqn}:
\[ 13 > (5-X) \sum_i^X  \M(\alpha_i) + \sum_{1 \leq i<j\leq X} \M(\alpha_i - \alpha_j) =S.  \]

\subsection{\texorpdfstring{$X=1$}{X=1}}

In this case, $\beta \in \field{84}$. We can write $\beta = \alpha + \zeta_4 \gamma$ with $\alpha, \gamma \in \field{21}$. We know that $\M(\alpha) + \M(\gamma) < \infrac{13}{4}$, so we may assume \WLOG that $\M(\alpha) \leq \infrac{13}{8}$. Then $\alpha$ is a root of unity. But then $\M(\gamma) < \infrac{9}{4}$ and $\Num(\gamma) \geq 6$, a contradiction by~\cite{Calegari} 7.0.5.

\subsection{\texorpdfstring{$X=2$}{X=2}}

$\M(\alpha_i) \geq \infrac{23}{6}$ contradicts equation~\ref{134eqn}. So by~\cite{Calegari} 7.0.9, $\Num(\alpha_i) \leq 5$. 

In the following table and throughout we list lower  bounds on the values of $\M$ and $S$. In all cases, $S\geq13$, contradicting equation~\ref{134eqn}.

\[
\begin{array}{cc|cc|c|c}
\multicolumn{2}{c|}{\Num(\alpha_i)}&\multicolumn{2}{|c|}{\M(\alpha_i)} &\multicolumn{1}{|c|}{\M(\alpha_1-\alpha_2)}&\multicolumn{1}{|c}{S}  \\\hline
\geq 2& 5& \tabfrac{5}{3} &\tabfrac{17}{6}&  2& \tabfrac[15]{1}{2}\\
\geq 3& 4&  2 &  \tabfrac{5}{2} &  1 &  \tabfrac[14]{1}{2}
\end{array}
\]

\subsection{\texorpdfstring{$X=3$}{X=3}}
The column $\M(\alpha_i-\alpha_j)$ is listed in the order $\alpha_1-\alpha_2, \alpha_1-\alpha_3, \alpha_2-\alpha_3$.
\[
\begin{array}{ccc|ccc|ccc|c}
\multicolumn{3}{c|}{\Num(\alpha_i)}&\multicolumn{3}{|c|}{\M(\alpha_i)} &\multicolumn{3}{|c|}{\M(\alpha_i-\alpha_j)}&\multicolumn{1}{|c}{S}  \\
\hline
1& 1& \geq 5&1&1&\tabfrac{17}{6}&0&2&2&\tabfrac[13]{2}{3}\\
1& 2& \geq 4&1&\tabfrac{5}{3}&\tabfrac{5}{2}&1&2&\tabfrac{5}{3}&15\\
1& \geq 3& \geq 3&1&2&2&\tabfrac{5}{3}&\tabfrac{5}{3}&0&\tabfrac[13]{1}{3}\\
2&2&3& \tabfrac{5}{3}&\tabfrac{5}{3}&2&0&1&1&\tabfrac[12]{2}{3}^*\\
2&2&\geq 4& \tabfrac{5}{3}&\tabfrac{5}{3}& \tabfrac{5}{2}& 0&\tabfrac{5}{3}&\tabfrac{5}{3} & 15 \\
2& \geq 3& \geq 3&\tabfrac{5}{3}&2&2&1&1&0&\tabfrac[13]{1}{3}\\
\geq 3& \geq 3& \geq 3&2&2&2&1^{**}&0&0&13\\
\end{array}
\]
\begin{description}
\item[$^*$] See that $\M(\alpha_1 - \alpha_2)=0$ and $\M(\alpha_3 - \alpha_1)=1$, or else $S>13$. But then $\beta$ can be written as a sum of 5 roots of unity: take $\other'_i = \other_i-\alpha_1$.
\item[$^{**}$] This results from assuming at least one pair is different. If all $\alpha_i$ are equal, then $\beta$ can be represented with $X=2$ by taking $\other'_j = \other_j - \alpha_1$.
\end{description}

\subsection{\texorpdfstring{$X=4$}{X=4}}

No two $\alpha_i$ are equal. If $\alpha_j=\alpha_k$, then there is another representation with $X<4$ given by $\other'_i = \other_i-\alpha_j$ for all $i$.

The column $\M(\alpha_i-\alpha_j)$ is listed in the order $\alpha_1-\alpha_2, \alpha_1-\alpha_3, \alpha_1-\alpha_4, \alpha_2-\alpha_3, \alpha_2-\alpha_4, \alpha_3-\alpha_4$.

\[
\begin{array}{cccc|cccc|cccccc|c}
\multicolumn{4}{c|}{\Num(\alpha_i)}&\multicolumn{4}{|c|}{\M(\alpha_i)} &\multicolumn{6}{|c|}{\M(\alpha_i-\alpha_j)}&\multicolumn{1}{|c}{S}  \\
\hline
1&1&1&\geq 4&1&1&1&\tabfrac{5}{2}&1&1&2&1&2&2&\tabfrac[14]{1}{2}\\
1&1&2& \geq 3&1&1&\tabfrac{5}{3}&2&1&1&\tabfrac{5}{3}&1&\tabfrac{5}{3}&1&13\\
1&1& \geq 3&\geq 3&1&1&2&2&1&\tabfrac{5}{3}&\tabfrac{5}{3}&\tabfrac{5}{3}&\tabfrac{5}{3}&1&\tabfrac[14]{2}{3}\\
1& 2& 2& 2&1&\tabfrac{5}{3}&\tabfrac{5}{3}&\tabfrac{5}{3}&1&1&1&1&1&1&12^\dagger\\
1&\geq 2& \geq 2&\geq 3&1&\tabfrac{5}{3}&\tabfrac{5}{3}&2&1&1&\tabfrac{5}{3}&1&1&1&13\\
2& 2& 2& 2&\tabfrac{5}{3}&\tabfrac{5}{3}&\tabfrac{5}{3}&\tabfrac{5}{3}&\tabfrac{5}{3}^{\dagger\dagger}&1&1&1&1&1&\tabfrac[13]{1}{3}\\
\geq 2& \geq 2& \geq 2& \geq 3 &\tabfrac{5}{3}&\tabfrac{5}{3}&\tabfrac{5}{3}&2&1&1&1&1&1&1&13
\end{array}
\]
\begin{description}
\item[$^\dagger$] If any of the differences is more than a single root of unity, it increases $S$ by at least $\infrac{2}{3}$, so at most one difference is more than a single root of unity. Thus we may assume \WLOG that $\M(\alpha_2 - \alpha_1) = \M(\alpha_3 - \alpha_1) = 1$. Then $\Num(\beta) \leq 6$, as evidenced by $\other'_i = \other_i-\alpha_1$ for all $i$.
\item[$^{\dagger\dagger}$] If every difference is a single root of unity, $\Num(\beta) = 5$: put $\other'_i = \other_i-\alpha_1$ for all $i$.
\end{description}

\subsection{\texorpdfstring{$X=5$}{X=5}}
This is not minimal, there is always a representation with $X < p$: put $\other'_i = \other_i-\alpha_1$ for all $i$.

\bibliographystyle{amsalpha}
\bibliography{bib}

\providecommand{\bysame}{\leavevmode\hbox to3em{\hrulefill}\thinspace}
\providecommand{\MR}{\relax\ifhmode\unskip\space\fi MR }
\providecommand{\MRhref}[2]{%
  \href{http://www.ams.org/mathscinet-getitem?mr=#1}{#2}
}
\providecommand{\href}[2]{#2}
\begin{thebibliography}{CMS11}

\bibitem[AH99]{AH}
Marta Asaeda and Uffe Haagerup, \emph{Exotic subfactors of finite depth with
  {J}ones indices {$(5+\sqrt{13})/2$} and {$(5+\sqrt{17})/2$}}, Comm. Math.
  Phys. \textbf{202} (1999), no.~1, 1--63. \MR{MR1686551 (2000c:46120)}

\bibitem[Cas69]{Cassels}
J.~W.~S. Cassels, \emph{On a conjecture of {R}. {M}. {R}obinson about sums of
  roots of unity}, J. Reine Angew. Math. \textbf{238} (1969), 112--131.

\bibitem[CMS11]{Calegari}
Frank Calegari, Scott Morrison, and Noah Snyder, \emph{Cyclotomic integers,
  fusion categories, and subfactors}, Comm. Math. Phys. \textbf{303} (2011),
  no.~3, 845--896.

\bibitem[IJMS]{J2}
Masaki Izumi, Vaughan F.~R. Jones, Scott Morrison, and Noah Snyder,
  \emph{Subfactors of index less than $5$, part $3$: quadruple points}, Comm.
  Math. Phys., to appear.

\bibitem[Jon68]{Jones1}
A.~J. Jones, \emph{Sums of three roots of unity}, Proc. Camb. Phil. Soc.
  \textbf{64} (1968), 673--682.

\bibitem[Jon69]{Jones2}
\bysame, \emph{Sums of three roots of unity. {II}}, Proc. Cambridge Philos.
  Soc. \textbf{66} (1969), 43--59. \MR{0238802 (39 \#166)}

\bibitem[Jon83]{J1}
Vaughan F.~R. Jones, \emph{Index for subfactors}, Invent. Math. \textbf{72}
  (1983), no.~1, 1--25. \MR{MR696688 (84d:46097)}

\bibitem[Kro37]{Kronecker}
L~Kronecker, \emph{Zwei s{\"a}tze {\"u}ber gleichungen mit ganzzahligen
  coefficient}, J. Reine Angew. Math. \textbf{53} (1837), 173--175.

\bibitem[Lox72]{Lox}
J.~H. Loxton, \emph{On the maximum modulus of cyclotomic integers}, Acta Arith.
  \textbf{22} (1972), 69--85. \MR{MR0309896 (46 \#9000)}

\bibitem[Man65]{Mann}
Henry~B. Mann, \emph{On linear relations between roots of unity}, Mathematika
  \textbf{12} (1965), 107--117.

\bibitem[PR98]{Poonen}
Bjorn Poonen and Michael Rubinstein, \emph{The number of intersection points
  made by the diagonals of a regular polygon}, SIAM Journal on Discrete
  Mathematics \textbf{11} (1998), 135--156.

\bibitem[Rob65]{Robinson}
Raphael~M. Robinson, \emph{Some conjectures about cyclotomic integers},
  Mathematics of Computation \textbf{19} (1965), 210--217.

\bibitem[Sch66]{Schinzel}
A.~Schinzel, \emph{On sums of roots on unity. {S}olution of two problems of
  {R}. {M}. {R}obinson}, Acta Arith. \textbf{11} (1966), 419--432.

\end{thebibliography}

%
%
%
%

\end{document}